\documentclass[12pt,reqno]{amsart}

\usepackage[textheight=21.8cm,textwidth=15.2cm,headsep=1cm,centering]{geometry}
\usepackage{color}
\usepackage{esint,amssymb}
\usepackage{graphicx}
\usepackage{tikz}

\usepackage{dsfont}

\usepackage{hyperref}

\allowdisplaybreaks

\newtheorem{theorem}{Theorem}
\newtheorem{proposition}[theorem]{Proposition}
\newtheorem{lemma}[theorem]{Lemma}
\newtheorem{corollary}[theorem]{Corollary}

\theoremstyle{definition}
\newtheorem{remark}[theorem]{Remark}

\newtheorem{definition}[theorem]{Definition}

\numberwithin{equation}{section}
\numberwithin{theorem}{section}

\newcommand{\R}{\mathbb{R}}

\newcommand{\average}{{\mathchoice {\kern1ex\vcenter{\hrule height.4pt
width 6pt depth0pt} \kern-9.7pt} {\kern1ex\vcenter{\hrule
height.4pt width 4.3pt depth0pt} \kern-7pt} {} {} }}

\renewcommand{\tilde}{\widetilde}

\newcommand{\Pc}{\mathcal{P}}
\newcommand{\Lc}{\mathcal{L}}
\newcommand{\Cc}{\mathcal{C}}
\newcommand{\sign}{\text{sign}}
\newcommand{\cara}{\mathds{1}}

\begin{document}

\title[Calibrations and null-Lagrangians for nonlocal perimeters]
{Calibrations and null-Lagrangians for \\ nonlocal perimeters and an application \\ to the viscosity theory}

\begin{abstract}
For nonnegative even kernels $K$, we consider the $K$-nonlocal perimeter functional acting on sets.
Assuming the existence of a foliation of space made of solutions of the associated $K$-nonlocal mean curvature
equation in an open set $\Omega\subset\R^n$, we built a calibration for the nonlocal perimeter inside $\Omega\subset\R^n$. 
The calibrating functional is a nonlocal null-Lagrangian. 
As a consequence, we conclude the minimality in $\Omega$ of each leaf of the foliation.
As an application, we prove the minimality of $K$-nonlocal minimal graphs and that they are the unique minimizers
subject to their own exterior data.

As a second application of the calibration, we give a simple proof of an important result from the seminal paper of 
Caffarelli, Roquejoffre, and Savin, stating that minimizers of the fractional perimeter are viscosity solutions. 
\end{abstract}

\author[Xavier Cabr\'e]{Xavier Cabr\'e}
\address{﻿X.C.\textsuperscript{1,2,3} ---
\textsuperscript{1}ICREA, Pg.\ Lluis Companys 23, 08010 Barcelona, Spain \& 
\textsuperscript{2}Universitat Polit\`ecnica de Catalunya, Departament de Matem\`{a}tiques, 
Diagonal 647, 08028 Barcelona, Spain \& 
\textsuperscript{3}BGSMath, Campus de Bellaterra, Edifici C, 08193 Bellaterra, Spain.
}
\email{xavier.cabre@upc.edu}


\thanks{X.C. is supported by grants MTM2017-84214-C2-1-P and MdM-2014-0445 (Government of Spain), 
and is a member of the research group 2017SGR1392 (Government of Catalonia).}

\maketitle

\section{Introduction}

Given a measurable function $K$ in $\R^ n$, a bounded open set~$\Omega \subset \R^{n}$, and a measurable set~$E \subset \R^{n}$, 
the \emph{$K$-nonlocal perimeter} of $\Omega$ (in all of $\R^n$) is defined by
$$
\Pc(\Omega):=\Lc(\Omega,\Omega^c),
$$
while the $K$-nonlocal perimeter of~$E$ inside~$\Omega$ is
\begin{equation}\label{Kper}
\Pc_\Omega (E) := \Lc (E \cap \Omega,E^c\cap\Omega) +  \Lc (E \cap \Omega,E^c\cap\Omega^c) + \Lc (E^c\cap\Omega, E \cap \Omega^c),
\end{equation}
where $A^c=\R^n\setminus A$ denotes the complement of a set and
\begin{equation}\label{Lper}
\Lc (A, B) := \int_A \int_B K(x - y) \, dy\, dx
\end{equation}
for any two disjoint measurable sets~$A, B \subset \R^{n}$. For the kernel $K$ we assume that
\begin{equation}\label{kernel}
K(z)\geq 0, \quad K(-z)=K(z),\quad \text{ and } \int_{\R^n} \min(1,|z|) K(z)\, dz <\infty.
\end{equation}
This is an extension of the fractional perimeter introduced by Caffarelli, Roquejoffre, and Savin \cite{CRS10}, 
in which $K(z)=|z|^{-n-\alpha}$ for some $\alpha\in (0,1)$.

If it happens that~$\Pc_\Omega (E) \le \Pc_\Omega (F)$
for every measurable set~$F \subset \R^{n}$ satisfying $F \setminus \Omega = E \setminus \Omega$, we then call~$E$ a~\emph{minimizer} 
of the~$K$-nonlocal perimeter in~$\Omega$.

Given the outside datum $E\setminus\Omega$, if a minimizer~$E$ of the $K$-nonlocal perimeter inside~$\Omega$ exists and if it is regular enough,
then it is easy to verify that it satisfies the Euler-Lagrange equation
\begin{equation} \label{H=0}
H_K[E](x) = 0
\end{equation}
at any point~$x \in  \partial E \cap\Omega$. Here $H_K[E](x)$ denotes the so-called~\emph{$K$-nonlocal mean curvature} of~$E$ at a 
point~$x$ of its boundary, and is defined by
\begin{equation} \label{Hdef}
 \begin{split}
H_K[E](x) & := \int_{\R^{n}} \left( \cara_{E^c}(y) - \cara_{E}(y)\right) K(x - y)\, dy\\
 & \; = \lim_{\varepsilon\downarrow 0} \int_{\R^{n}\setminus B_\varepsilon(x)} \left( \cara_{E^c}(y) - \cara_{E}(y)\right) K(x - y)\, dy,
 \end{split}
\end{equation}
where $\cara$ denotes the characteristic function and the integral is meant in the Cauchy principal value sense ---the second expression
in \eqref{Hdef}--- if $K$ is not integrable about the origin.

We wish to give sufficient conditions for a set with zero $K$-nonlocal mean curvature in $\Omega$ to be a minimizer.
For local elliptic functionals, it is well known that if $\Omega$ is foliated by disjoint hypersurfaces which satisfy, all of them, the
Euler-Lagrange equation, then each of them is a minimizer subject to its own boundary conditions. This is closely related
to the Weierstrass field theory. In fact, given one leaf it suffices to have
sub and super solutions on each side, respectively, of the given leaf ---this is extremely useful for applications.
There is a simple argument based on the strong maximum principle to prove this fact 
whenever an existence theorem of regular enough minimizers is available ---see for instance the 
``Alternative proof of Theorem 1.8'' in the survey paper \cite{CP} by G. Poggesi and the author where this is done 
in the context of classical minimal surfaces. This proof gives in fact more information: the uniqueness of solution 
with boundary datum equal to that of the given leaf. The argument also works in the 
nonlocal setting since the operator $H_K$ enjoys the maximum principle, but it requires to know the existence of a minimizer 
(for given exterior datum) and some regularity for it. Unfortunately,  these tools are not available for many kernels $K$ ---for instance,
if $K\in L^1(\R^n)$ one may not have compactness for sequences of
sets with uniformly bounded $K$-perimeter\footnote{ However, a remarkable existence result of
minimizers for very general kernels has been proved in \cite{CSV16} as a consequence of a delicate 
a priori BV-estimate established in the same paper.}. In addition, the regularity of the minimizers is in
general not available.

It is therefore of interest to extend the theory of  \emph{calibrations}, from the Calculus of Variations,
to the nonlocal setting.
In the local context, the calibration is a functional that can be constructed in the presence of a foliation made of sub and super solutions.
It allows then to prove that the given leaf of the foliation is automatically a minimizer
for its own boundary datum. See for instance \cite[Proof of Theorem 1.8]{CP}, and references within it,
for this program in the context of classical minimal 
surfaces. Once the calibration is available, to conclude minimality there is no need
to have an existence result of minimizers, neither to know their regularity.

In this article we find a calibration for the $K$-nonlocal perimeter in the presence of a foliation made of sub and super solutions 
of \eqref{H=0}. When the foliation is given by solutions, then the calibration functional is a \emph{nonlocal null-Lagrangian}. 
More precisely, note first that the $K$-nonlocal perimeter of a set $F\subset \R^n$ inside $\Omega$ can be rewritten as
\begin{equation*}\label{perime}
\Pc_\Omega (F)= \int_{(\R^n\times\R^n)\setminus(\Omega^c\times\Omega^c)} |\cara_F(x)-\cara_F(y)| \,K(x-y) \,dx\, dy.
\end{equation*}
Given a measurable function $\phi_E$ in $\R^{n}$, we define the calibration, acting on sets $F$ 
satisfying $F\setminus\Omega=E\setminus\Omega$, by
\begin{equation*}\label{cal}
\Cc_\Omega (F) := \int_{(\R^n\times\R^n)\setminus(\Omega^c\times\Omega^c)} \sign(\phi_E(x)-\phi_E(y)) 
\left( \cara_F(x)-\cara_F(y)\right) K(x-y) \,dx\, dy.
\end{equation*}
It will serve to establish that a set $E$ is a minimizer in $\Omega$ whenever there exists a function 
$\phi_{E}$ verifying $E=\{\phi_{E}>0\}$ and whose levels sets are subsolutions 
of the $K$-nonlocal mean curvature equation in $\Omega\cap E$ and supersolutions in $\Omega\setminus E$.
If they are solutions in all of $\Omega$, then the calibration will be a null-Lagrangian, that is, we will have
$$
\Cc_{\Omega}(F)=\Cc_{\Omega}(E) \; \text{ whenever } F\setminus\Omega=E\setminus \Omega.
$$ 
All this is presented in Section~2, where one can find an alternative expression for the calibration:
\begin{equation*}
\Cc_\Omega (F) =  \; 2\int_{F\cap\Omega} H_K[\phi_E](x)\, dx 
+2 \int_{E\setminus\Omega} dx\int_{\Omega}dy \;\, {\rm sign}(\phi_E(x)-\phi_E(y)) K(x-y);
\end{equation*} 
see \eqref{nmcfol} for the definition of the ``level set nonlocal mean curvature'' $H_K[\phi_E](x)$
and note that, interestingly, the last term in the previous equality is independent of the competitor $F$.
In Section~2 we are also concerned with one-sided minimizers, as well as with the following application.

We use the calibration to establish that subgraphs
$$
E = \left\{ (x', x_{n}) \in \R^{n-1} \times \R : x_{n} < u(x') \right\}
$$
having zero $K$-nonlocal mean curvature ---as well as a certain regularity--- are minimizers in
every bounded set of $\R^n$ and, furthermore, are the unique minimizers subject to their own exterior datum. 
This fact had been already proved, using other tools, by Cozzi and Lombardini~\cite{L,CL17}. Their results are for the standard power kernel  
but under more general hypothesis. They also establish a more general result giving an equivalence between different notions
for a graph to be an $\alpha$-nonlocal minimal surface (solutions in the viscosity, pointwise, or weak sense, as well as minimizers).
Also for the standard power kernel, a further development has been done in \cite{CC} by Cozzi and the author, where it is 
established that $\alpha$-nonlocal minimal graphs are always smooth.
Another interesting result is that of Dipierro, Savin, and Valdinoci~\cite{DSV16} where it is proved, also for the power kernel, that
if the exterior datum (outside of a cylinder) is a graph then the minimizer is also a graph inside.

Another motivation for our work comes from singular cones.
Within the theory of nonlocal minimal surfaces, an important open problem concerns the Simons cone in $\R^{n}=\R^{2m}$,
$$
\{ x\in\R^{2m}\, : \, x_1^2+\ldots +x_m^2 < x_{m+1}^2
+\ldots +x_{2m}^2 \}.
$$
It has zero $\alpha$-nonlocal mean curvature for every $m$, but it is not known to be a minimizer
in any even dimension ---it is expected to be a minimizer in dimensions $2m \geq 8$ as in the classical case; see \cite{DDPW14}.
To built a foliation, in certain dimensions, made of sub and super solutions 
of the nonlocal mean curvature equation, and having the Simons cone as one leaf, 
is a very challenging open problem. 
If available, then our result would conclude the minimality of the Simons cone in those dimensions.
See the survey \cite{CP} for a well known foliation in the case of classical minimal surfaces.

Section 3 uses the same method to give a relation for the difference of the $K$-nonlocal perimeter of two ordered sets
$F\subset E$. It will be useful for our application to the viscosity theory, for which we need to consider a generalized notion 
$\tilde H_K$ of $K$-nonlocal mean curvature that is well defined for all sets, even irregular ones. The relation is very simple and reads 
\begin{equation*}
\Pc_\Omega (E) \leq \Pc_\Omega (F)+ \int_{E\setminus F}  \tilde H_K [\phi_E](x)\, dx,
\end{equation*}
where $\phi_{E}$ is any function taking the value $+\infty$ in $F$ and $-\infty$ in $E^{c}$ and where, 
in a smooth setting, we would have
$\tilde H_K [\phi_E](x)=H_K [\{\phi_E>\phi_E(x)\}](x)$. After completing our work, we have been told
that some results closely 
related to this inequality have already appeared in the literature; see Remark~\ref{rk:ineq} for more details.

In Section 4 we are concerned with an important result of Caffarelli, Roquejoffre, and Savin,
Theorem~5.1 of~\cite{CRS10}. It established that minimizers of the fractional perimeter are viscosity solutions. 
We use the calibration to give a simple proof of it, which may be of interest since
their proof was clever but quite involved. In addition, we extend
this viscosity result to general kernels satisfying a power upper bound\footnote{ When $K\in L^1(\R^n)$, 
minimizers satisfy the Euler-Lagrange equation pointwise, as it was
shown by Maz\'on, Rossi, and Toledo~\cite{MRT}.}.

In independent work from ours, Valerio Pagliari \cite{P} finds the analytical properties that  
a ``nonlocal field'' (i.e., a scalar function $\zeta=\zeta(x,y)$ with doubled variables $x$
and $y$) must satisfy to guarantee that a certain new functional is a calibration for a given set. 
In~\cite{P} such a ``nonlocal field'' is built for halfspaces ---a result allowing the author to establish 
that halfspaces are unique minimizers and, from this, an 
interesting $\Gamma$-convergence result concerning the scaling limit of the $K$-nonlocal perimeter functional.
On the other hand, \cite{P} does not built a ``nonlocal field'', nor a calibration, in the presence of a foliation by solutions ---as
done in the current paper. Finally, to relate both articles, it is easy to see that if a foliation as in our Definition~\ref{def:calib}
makes the inequalities \eqref{signs-int} to be equalities, then the ``nonlocal field''
$\zeta (x,y) := \sign (\phi_E(x) -\phi_E(y))$ is a calibration for $\cara_E$ in the sense of Pagliari~\cite{P}.

\section{The calibration. Minimality}

Let us start defining what we mean by a ``foliation'' made of sub and super solutions. The leaves of the foliation
will be the level sets of a certain function $\phi_E$. 

Let $K$ satisfy \eqref{kernel} and $\phi_E :\R^ n\to\R$ be a measurable function.
For $x\in\R^n$, we define
\begin{equation}\label{nmcfol}
 H_{K} [\phi_E](x):=\lim_{\varepsilon\downarrow 0}\int_{\R^n\setminus B_{\varepsilon}(x)}
 \sign(\phi_E(x)-\phi_E(y))\, K(x-y)\, dy
\end{equation}
whenever the limit exists, where we set
\begin{equation*}
\sign (t) =
\begin{cases}
-1 & \quad \mbox{if } t<0,\\
0 & \quad \mbox{if } t=0,\\
1 & \quad \mbox{if } t > 0.
\end{cases}
\end{equation*}

\begin{definition} \label{def:calib}
Given a bounded open set $\Omega\subset\R^n$ and a measurable set $E \subset\R^n$, we say that
``$\,\Omega$ is foliated by sub and super solutions adapted to $E$'' whenever there exists a measurable function
$\phi_E:\R^n\to \R$ such that:
\begin{itemize}
\item[(i)] $E=\{\phi_E>0\}:=\{x\in\R^n : \phi_E(x)>0\}$ up to a zero measure set;
\item[(ii)] the limit  in \eqref{nmcfol} exists for a.e.\ $x\in\Omega$ and the functions of $x\in\Omega$ given by the
integrals in the right-hand side of~\eqref{nmcfol} converge in $L^1(\Omega)$ to $H_K[\phi_E]$, as 
$\varepsilon\downarrow 0$.\footnote{ When $K\not\in L^1(\R^n)$, this hypothesis imposes some
regularity on the level sets of $\phi_E$. It is needed for the
proofs in this section to work. An analogue regularity assumption on the
``nonlocal field'' $\zeta (x,y)$ is also needed in the
article~\cite{P}, even if not explicitely stated there, to
ensure that all quantities in the Proof of Theorem 2 in~\cite{P} ---such as the
second expression for $a(v)$--- are well defined.}
\item[(iii)] it holds that
\begin{equation}\label{signs-int}
H_{K} [\phi_E](x)
\begin{cases}
\leq 0 & \quad \mbox{for a.e.\ } x\in\Omega\cap E,\\
\geq 0 & \quad \mbox{for a.e.\ } x\in\Omega \setminus E.
\end{cases}
\end{equation}
\end{itemize}
\end{definition}

\begin{remark}\label{rk:fat}
Some comments are in order. When 
$\{\phi_E = \phi_E(x)\}$ is a smooth hypersurface separating the $\phi_E(x)$-sub and super level sets of $\phi_E$,
then \eqref{nmcfol} coincides
with the $K$-nonlocal mean curvature  at $x$ of the $\phi_E(x)$-super level set $\{\phi_E>\phi_E(x)\}$, as defined in \eqref{Hdef}.

If all level sets of $\phi_E$ are smooth hypersurfaces in a neighborhood of $\overline\Omega$ and $\Omega$ is also smooth, 
then hypothesis~(ii) will hold ---see Footnote~3 for more comments on hypothesis (ii).

Finally, even if this is not the case in our applications or those that we have in mind, one could wonder
about the situation in which $\phi_E$ had a level set with positive measure (or even with nonempty interior) within $\Omega$.
It is important to notice that, in this degenerate case, \eqref{signs-int} would impose a strong geometric condition.
Indeed, the inequalities \eqref{signs-int}
should also hold at the points of the fat level set ---note that $H_K[\phi_E](x)$ would be surely well defined and finite at interior 
points of the level~set.
\end{remark}

Notice that we can rewrite the $K$-nonlocal perimeter of a set
$F\subset \R^n$ inside $\Omega$, that we assume to be finite, as
\begin{equation}\label{per2}
\Pc_\Omega (F)= \int_{(\R^n\times\R^n)\setminus(\Omega^c\times\Omega^c)} |\cara_F(x)-\cara_F(y)| \,K(x-y) \,dx\, dy <\infty.
\end{equation}
Given the function $\phi_E$ we can now define a new functional, called the calibration and acting on sets $F$ 
satisfying $F\setminus\Omega=E\setminus\Omega$, by
\begin{equation}\label{cal1}
\Cc_\Omega (F) := \int_{(\R^n\times\R^n)\setminus(\Omega^c\times\Omega^c)} \sign(\phi_E(x)-\phi_E(y)) 
\left( \cara_F(x)-\cara_F(y)\right) K(x-y) \,dx\, dy.
\end{equation}
The integrand defines an integrable function thanks to \eqref{per2}.

Note that, as a consequence of Definition \ref{def:calib} (i),
\begin{equation}\label{basic-cal}
 \Pc_\Omega (F)\geq \Cc_\Omega (F) \; \text{ for all sets $F$,}\text{ while}\;\; \Pc_\Omega (E)=\Cc_\Omega (E).
\end{equation}

The following is another useful expression for the calibration. It requires the regularity assumption (ii) in the previous definition
---but not the key hypothesis \eqref{signs-int}.

\begin{proposition}\label{prop:formula}
Let $K$ satisfy \eqref{kernel}. Given a bounded open set $\Omega\subset\R^n$ satisfying $\Pc(\Omega)<\infty$
and a measurable set $E \subset\R^n$, assume that (i) and (ii) in Definition~\ref{def:calib} hold for some
measurable function $\phi_{E}:\R^{n}\to\R$.

Then, for every measurable set such that $F\setminus\Omega=E\setminus\Omega$ and $\Pc_{\Omega}(F)<\infty$, we have
\begin{equation}\label{cal2}
\Cc_\Omega (F) =  \; 2\int_{F\cap\Omega} H_K[\phi_E](x)\, dx 
+2 \int_{E\setminus\Omega} dx\int_{\Omega}dy \;\, {\rm sign}(\phi_E(x)-\phi_E(y)) K(x-y).
\end{equation}
\end{proposition}

For this expression to hold we emphasize that it is important to have $F\setminus\Omega=E\setminus\Omega$. Note that
the second term in the right hand side of \eqref{cal2} is independent of the competitor $F$.

The previous proposition will be established within the proof of the following theorem, which is our main result.

\begin{theorem} \label{thm:main}
Let $K$ satisfy \eqref{kernel}. Given a bounded open set $\Omega\subset\R^n$ satisfying $\Pc(\Omega)<\infty$
and a measurable set $E \subset\R^n$, assume that
``$\,\Omega$ is foliated by sub and super solutions adapted to $E$'' as in Definition~\ref{def:calib}.
We then have:

{\rm (a)} $E$ is a minimizer of the $K$-nonlocal perimeter in $\Omega$. 

{\rm (b)}  Assume in addition that the kernel $K>0$ in $\R^n$,
the function $\phi_E$ is continuous, $\{\phi_E=0\}\cap\Omega$ has zero measure, and 
both $E\setminus \overline B_R$ and $(\overline E)^c\setminus \overline B_R$ are nonempty for some ball $B_R\supset \Omega$. 
Then, $E$ is the unique minimizer in $\Omega$ subject to its exterior datum.

{\rm (c)} Assume in addition that ``$\,\Omega$ is foliated by solutions adapted to $E$'', i.e., that 
the conditions in  Definition \ref{def:calib} hold and that the inequalities in \eqref{signs-int} are both equalities. Then
the functional $\Cc$ is a nonlocal null-Lagrangian in $\Omega$, that is,
\begin{equation*}\label{nullL}
\Cc_\Omega (F)=\Cc_\Omega (E) \; \text{ for all sets $F$ such that } F\setminus\Omega=E\setminus\Omega.
\end{equation*}
\end{theorem}

For the proof, and also for next sections, it will be useful to introduce the $L^{1}(\R^n)$ kernel
\begin{equation}\label{epsker}
K_\varepsilon(z) := \cara_{\R^{n}\setminus B_{\varepsilon}(0)}(z) K(z), \quad z\in\R^n,
\end{equation}
as well as the $K_{\varepsilon}$-nonlocal mean curvature
\begin{equation}\label{epsnmc}
H_{K_{\varepsilon}} [\phi_E](x):=\int_{\R^n} \sign(\phi_E(x)-\phi_E(y))\, K_{\varepsilon}(x-y)\, dy.
\end{equation}
It converges, by definition, to $H_{K}[\phi_{E}](x)$ as $\varepsilon\downarrow 0$ whenever the limit exists.

\begin{proof}[Proof of Theorem \ref{thm:main}]
(a) Thanks to \eqref{basic-cal}, to verify that $E$ is a minimizer it suffices to show that
\begin{equation}\label{ineq-cal}
\Cc_\Omega (F) \geq \Cc_\Omega (E)  \; \text{ for every set $F$ such that } F\setminus\Omega=E\setminus\Omega
\end{equation}
and with $\Pc_\Omega(F)<\infty$. Hence, recalling \eqref{per2} and \eqref{cal1}, we have
\begin{equation}\label{caleps}
\Cc_\Omega (F) = \lim_{\varepsilon\downarrow 0} \int_{(\R^n\times\R^n)\setminus(\Omega^c\times\Omega^c)} \sign(\phi_E(x)-\phi_E(y)) 
\left( \cara_F(x)-\cara_F(y)\right) K_\varepsilon (x-y) \,dx\, dy
\end{equation}
by dominated convergence.

We first establish expression \eqref{cal2} for the calibration. Note that all the following integrals are convergent since
$K_\varepsilon\in L^ 1(\R^n)$. Since the functions $K_\varepsilon$ and $\sign$ 
are even and odd, respectively, we have that
\begin{equation*}
\begin{split}
& \hspace{-.5cm} \int_{(\R^n\times\R^n)\setminus(\Omega^c\times\Omega^c)} \sign(\phi_E(x)-\phi_E(y)) 
\left( \cara_F(x)-\cara_F(y)\right) K_\varepsilon (x-y) \,dx\, dy\\
& = 2 \int_{(\R^n\times\R^n)\setminus(\Omega^c\times\Omega^c)} \sign(\phi_E(x)-\phi_E(y)) 
\cara_F(x) K_\varepsilon (x-y) \,dx\, dy.
\end{split}
\end{equation*}
We now break this integral over $(\R^n\times\R^n)\setminus(\Omega^c\times\Omega^c)$ into two parts, taking
into account that $x$ must belong to $F$: 
when $x\in F\cap\Omega$ (in which case $y$ runs through all 
of $\R^n$) and when $x\in F\setminus\Omega =  E\setminus\Omega$ (in which case $y$ runs through $\Omega$).
Recalling \eqref{epsnmc}, we get that

\begin{align*}
& \hspace{-.4cm} 2 \int_{(\R^n\times\R^n)\setminus(\Omega^c\times\Omega^c)} \sign(\phi_E(x)-\phi_E(y)) 
\cara_F(x) K_\varepsilon (x-y) \,dx\, dy\\
& =\, 2\int_{F\cap\Omega} dx \int_{\R^n} dy\;\, \sign(\phi_E(x)-\phi_E(y)) \, K_\varepsilon (x-y)\\
& \hspace{.8cm} +2 \int_{E\setminus\Omega} dx\int_{\Omega}dy \;\, \sign(\phi_E(x)-\phi_E(y)) K_\varepsilon (x-y)\\
& =  \; 2\int_{F\cap\Omega} H_{K_\varepsilon}[\phi_E](x)\, dx 
+2 \int_{E\setminus\Omega}\, dx\int_{\Omega}dy \;\, \sign(\phi_E(x)-\phi_E(y)) K_\varepsilon (x-y).
\end{align*}
Note that the integrands in the last integral are dominated by $K(x-y)$, an integrable function over $(E\setminus\Omega)\times\Omega$ 
since $\Lc(E\setminus\Omega,\Omega)\leq \Lc(\Omega^c,\Omega)=\Pc(\Omega) <\infty$.
On the other hand, the first integral in the last
expression has limit as $\varepsilon\downarrow 0$
---since by hypothesis (ii) in Definition~\ref{def:calib}, $H_{K_\varepsilon}[\phi_E]$ converges in $L^1(\Omega)$ to $H_{K}[\phi_E]$.
Therefore, letting $\varepsilon\downarrow 0$, and recalling \eqref{caleps}, we conclude \eqref{cal2}.

Now, with \eqref{cal2} at hand, using that $(F\cap\Omega)\cup(E\setminus F)= (E\cap\Omega)\cup(F\setminus E)$
and that both unions are disjoint, we can express the first integral in \eqref{cal2} as 
\begin{equation}\label{null}
\int_{F\cap\Omega} H_{K}[\phi_E](x)\, dx = \left( \int_{E\cap\Omega} dx +\int_{F\setminus E} dx - \int_{E\setminus F} dx \right) 
H_{K}[\phi_E](x).
\end{equation}
Since both $F\setminus E$ and $E\setminus F$ are contained in $\Omega$, using the hypothesis \eqref{signs-int}
we conclude \eqref{ineq-cal}.

\medskip

(b) To prove uniqueness under the additional assumptions on $K$, $\phi_E$, and $E$, assume that $F$ is another minimizer. 
Then we must have equality in \eqref{basic-cal}.
Hence, looking at \eqref{per2} and \eqref{cal1} we deduce, since $K>0$, that 
\begin{equation}\label{uniq}
\phi_E(x) > \phi_E (y) \quad\text{a.e.\ in } \{(x,y)\in \left( (F\cap\Omega)\times F^c\right) \cup 
\left( F\times (F^c\cap\Omega)\right)\}.
\end{equation}

By hypothesis, $\phi_E$ takes both a positive value and a negative value in $(\overline B_R)^c$. Thus, for every $\delta>0$ small enough, 
both $\{-\delta<\phi_E<0\}\setminus \overline B_R$ and $\{0<\phi_E<\delta\}\setminus \overline B_R$ are nonempty open sets. Hence, 
since  $F^c\setminus \overline B_R=\{\phi_E\leq 0\}\setminus \overline B_R$ and $F\setminus \overline B_R=
\{\phi_E> 0\}\setminus \overline B_R$, from 
\eqref{uniq} we deduce (by letting $\delta\downarrow 0$) that 
$$
\phi_E(x)\geq 0 \;\text{ for  a.e.\ } x\in F\cap\Omega \quad\text{ and } \quad
\phi_E(y)\leq 0 \;\text{ for  a.e.\ } y\in F^c\cap\Omega.
$$
From this, since $\{\phi_E=0\}\cap\Omega$ has zero measure  by hypothesis, we deduce that $F\cap\Omega=\{\phi_E>0\}\cap\Omega
=E\cap\Omega$ up to a measure zero set.

\medskip 

(c) Finally, if the inequalities in \eqref{signs-int} are both equalities, the last two integrals in \eqref{null}
vanish, and thus \eqref{cal2} leads to $\Cc_{\Omega}(F)=\Cc_{\Omega}(E)$.
\end{proof}

As an immediate consequence we obtain the following result on nonlocal minimal graphs.

\begin{corollary}  \label{cor:graphs}
Let $K$ satisfy \eqref{kernel}. Given a bounded open set $\Omega\subset\R^n$ satisfying $\Pc(\Omega)<\infty$, let~$E$ be the subgraph
$$
E = \left\{ (x', x_n) \in \R^{n-1} \times \R : x_{n} < u(x') \right\}
$$
of a continuous function~$u: \R^{n-1} \to \R$. Assume that~$E$ has zero $K$-nonlocal mean curvature at every $x\in\partial E\cap\Omega$, defined by the principal value limit \eqref{Hdef}, which we assume to be uniform in $x\in\partial E\cap\Omega$.

Then, $E$ is a minimizer of the $K$-nonlocal perimeter in $\Omega$. If in addition $K>0$ in $\R^n$, 
then $E$ is the unique minimizer in $\Omega$ subject to its own exterior datum.
\end{corollary}

Note that halfspaces satisfy all the assumptions of the corollary ---since $H_{K_\varepsilon}\equiv 0$ on their boundary, 
for all $\varepsilon>0$. 

See our introduction for more comments on nonlocal minimal graphs.

\begin{proof}[Proof of Corollary \ref{cor:graphs}]
Taking
$$
\phi_E(x):=u(x')-x_n,
$$
we immediately deduce the result from Theorem \ref{thm:main} (a),(b). Note the invariance of \eqref{nmcfol} by vertical translations.
\end{proof}

The last result of this section deals with one-sided minimizers.

\begin{theorem}  \label{thm:oneside}
Assume the same hypotheses of Theorem \ref{thm:main} except for \eqref{signs-int}, in which 
we only assume the first inequality for a.e.\ $x\in \Omega\cap E$. Then $E$ is a one-sided minimizer in $\Omega$ from inside, that is,
\begin{equation}\label{one-sided}
 \Pc_\Omega(E)\leq \Pc_\Omega(F) \quad\text{for all $F$ satisfying } F\subset E \text{ and } F\setminus\Omega=E\setminus\Omega.
\end{equation}

If in addition the first inequality in \eqref{signs-int} is strict, then the inequality in \eqref{one-sided} is also strict,
unless $F=E$ up to a measure zero set.
\end{theorem}

\begin{proof}
The proof is the same as that of Theorem \ref{thm:main} noticing that in \eqref{null} the set $F\setminus E$ is empty. On the other hand,
since $E\setminus F \subset\Omega\cap E$, the last integral over $(E\setminus F)\times\R^n$ is nonpositive (or negative in the last assertion
of the theorem) by hypothesis.
\end{proof}

\section{A relation between the nonlocal perimeter of two ordered sets}

The following result gives a simple inequality for the difference of the $K$-nonlocal perimeters inside $\Omega$ of two ordered
sets $F\subset E$. The idea behind its proof is the same as in the calibration of the previous section. However, the new statement is adapted 
to our application to the viscosity theory in which, as we will see, apriori the foliation could be nonsmooth.
We therefore introduce the following notion of nonlocal mean curvature, which is always well defined (through a $\limsup$), 
even when the foliation is not smooth.

\begin{definition}\label{def:NMCgen}
Let $K$ satisfy \eqref{kernel} and $\phi_E :\R^ n\to [-\infty,+\infty]$ be a measurable function.
For $x\in\R^n$ with $\phi_E(x)\in\R$, we define
\begin{equation}\label{gennmc}
 \tilde H_{K} [\phi_E](x):=\limsup_{\varepsilon\downarrow 0} H_{K_\varepsilon} [\phi_E](x),
\end{equation}
where $K_\varepsilon$ and $H_{K_\varepsilon}$ are given by \eqref{epsker} and \eqref{epsnmc}.
\end{definition}

When $\{\phi_E = \phi_E(x)\}$ is a smooth hypersurface separating the $\phi_E(x)$-sub and super level sets of $\phi_E$,
then \eqref{gennmc} coincides
with the $K$-nonlocal mean curvature  at $x$ of the $\phi_E(x)$-super level set $\{\phi_E>\phi_E(x)\}$ as defined in \eqref{Hdef}.

\begin{theorem}  \label{formula}
Let $K$ satisfy \eqref{kernel} and $\Omega \subset \R^{n}$ be a bounded open set. Let $F$ and $E$
be two measurable sets of $\R^{n}$ such that $F\subset E$ and $F\setminus \Omega = E\setminus \Omega$. 
Given a measurable function $\phi : E\setminus F \to \R$, define
\begin{equation}\label{totalphi}
 \phi_E (x) := 
\begin{cases}
+\infty & \quad \mbox{if } x\in F,\\
\phi(x) & \quad \mbox{if } x\in E\setminus F,\\
-\infty & \quad \mbox{if } x\not\in E.
\end{cases}
\end{equation}
Assume that there exist constants $C_0\in\R$ and $\varepsilon_0>0$ such that
\begin{equation}\label{Hbdd}
H_{K_\varepsilon} [\phi_E](x) \le C_0\qquad\text{for a.e.\ } x\in E\setminus F \text{ and every } 
\varepsilon\in(0,\varepsilon_0).
\end{equation}

Then, if $\Pc_\Omega (F)<+\infty$ we have
\begin{equation}\label{key2}
\Pc_\Omega (E) \leq \Pc_\Omega (F)+ \int_{E\setminus F}  \tilde H_K [\phi_E](x)\, dx <+\infty.
\end{equation}
\end{theorem}

\begin{remark}\label{rk:ineq}
After completing our work, we have learnt that results closely related to inequality \eqref{key2}
have already appeared, under stronger regularity assumptions, in Proposition 4.25 of 
the work \cite{LMaster} by Lombardini, as well as in Proposition 4.5 of the
article \cite{CMP} by Chambolle, Morini, and Ponsiglione, also under stronger regularity 
assumptions but for a large class of generalized perimeters and curvatures. Both articles
give conditions to guarantee the reversed
inequality in \eqref{key2}, as well as to further guarantee equality in \eqref{key2} ---we are referring here
to the first inequality in~\eqref{key2}. 
They do not treat the validity of inequality \eqref{key2} under weaker 
regularity assumptions as ours.
\end{remark}

Two more comments are in order. First, the $\limsup$ in \eqref{gennmc} is bounded above by a constant since we assume \eqref{Hbdd}, but
it could be equal to $-\infty$ at some points $x\in\R^n$ since we do not assume regularity of the level sets of
$\phi_E$ ---this lack of regularity may happen in the application to the viscosity theory in next section, as it will
be pointed out there\footnote{ However, \eqref{key2} shows that we will necessarily have $\int_{E\setminus F}  \tilde H_K [\phi_E](x)\, dx >-\infty$,
since $\Pc_\Omega(E)\geq 0$ and $\Pc_\Omega(F)<+\infty$.}.

Our second comment is heuristic, but relevant to understand the geometry of the foliation given by the level sets of $\phi_E$
(in this respect, see also the comments made in Remark~\ref{rk:fat}).
Note that the theorem does not assume 
$E=\{\phi_{E}>0\}$, in contrast with the hypothesis in the previous section.
However, since $\phi_E\equiv -\infty$ in $E^c$ we have that $E^c\subset
\{ \phi_E > \phi_E(x)\}^{c}$, and hence points in $E^c$ contribute with a $1$ in the definition of $\tilde H_K [\phi_E](x)$. 
At the same time, we have $\tilde H_K [\phi_E](x)\leq C_0 <+\infty$ by \eqref{Hbdd}. 
Thus, heuristically, if the kernel $K$ is not integrable about the origin, 
we expect the level set of $\phi_E$ passing through $x\in\partial E$ (or the limit of interior level sets) to coincide with $\partial E$
---since it should be tangent to it at all points $x$ in order to make the $\limsup$-principal value in  \eqref{gennmc} less than
$+\infty$.

\begin{proof}[Proof of Theorem \ref{formula}] 
Given $\varepsilon>0$, 
we denote by $\Pc_{\varepsilon,\Omega}$ the $K_\varepsilon$-nonlocal perimeter inside $\Omega$ and by $\Lc_\varepsilon$ the
corresponding interaction functional 
defined by, respectively, \eqref{Kper} and \eqref{Lper} with $K$ replaced by $K_\varepsilon$.

Since $K_\varepsilon\in L^1(\R^n)$ and $E\setminus F\subset\Omega$ is bounded,
all the integrals in the following equalities are convergent. We have
\begin{align}
\Pc_{\varepsilon,\Omega} (E) -\Pc_{\varepsilon,\Omega}(F) &=   \Lc_\varepsilon (E\setminus F, E^c) - \Lc_\varepsilon (E\setminus F, F) \notag\\
& \hspace{-1.5cm} = \int_{E\setminus F} dx \int_{E^c} dy \, K_\varepsilon(x-y) - \int_{E\setminus F} dx \int_{F} dy \, K_\varepsilon(x-y) 
\notag\\
& \hspace{-1.5cm}   =  \int_{E\setminus F} dx \int_{E^c} dy \, K_\varepsilon(x-y) - \int_{E\setminus F} dx \int_{F} dy \, 
  K_\varepsilon(x-y)\notag\\
&  +\int_{E\setminus F} dx\int_{E\setminus F}dy \;\, \sign(\phi_E(x)-\phi_E(y))\, K_\varepsilon(x-y)\notag\\
& \hspace{-1.5cm} = \int_{E\setminus F} dx \left( \int_{E^c} dy + \int_{F} dy +\int_{E\setminus F}dy \right)
\sign(\phi_E(x)-\phi_E(y))\, K_\varepsilon(x-y)\notag\\
& \hspace{-1.5cm} = \int_{E\setminus F} dx \int_{\R^n} dy \,\; \sign(\phi_E(x)-\phi_E(y))\, K_\varepsilon(x-y)\notag\\
& \hspace{-1.5cm}  = \int_{E\setminus F}  H_{K_\varepsilon} [\phi_E](x)\, dx.\notag
\end{align}

We have $\Pc_{\varepsilon,\Omega} (F)\leq \Pc_\Omega(F)<+\infty$. Thus, from the previous equalities we deduce that
$$
\Pc_{\varepsilon,\Omega} (E) \leq \Pc_\Omega(F)
+\int_{E\setminus F}  H_{K_\varepsilon} [\phi_E](x)\, dx.
$$
We now take $\limsup$ in this inequality as $\varepsilon\downarrow 0$. In the right hand side, by \eqref{Hbdd} and Fatou's lemma 
applied to the nonnegative functions
$C_0-H_{K_\varepsilon} [\phi_E](x)$ (recall that $E\setminus F\subset\Omega$ is bounded), we have that
\begin{equation*}
\begin{split}
& \limsup_{\varepsilon\downarrow 0} \int_{E\setminus F}  H_{K_\varepsilon} [\phi_E](x)\, dx \\
& \hspace{1cm}\leq \int_{E\setminus F}  \limsup_{\varepsilon\downarrow 0}  H_{K_\varepsilon} [\phi_E](x)\, dx
=\int_{E\setminus F} \tilde H_{K} [\phi_E](x)\, dx.
\end{split}
\end{equation*}
On the other hand, by monotone convergence the left hand side tends to $\Pc_\Omega (E)$ as $\varepsilon\downarrow 0$, 
and hence we conclude \eqref{key2}. Note that the integral in \eqref{key2} is not $+\infty$ since the integrand is bounded 
by the constant $C_0$, thanks to \eqref{Hbdd}, and $E\setminus F\subset\Omega$ is bounded.
\end{proof}

\section{A simple proof that minimizers are viscosity solutions}

Throughout this section $B_r=B_r(0)\subset\R^{n}$ are balls centered at the origin. 
We use the notation $x=(x',x_n)\in\R^{n-1}\times \R$, $B_r'=\{x'\in\R^{n-1} : |x'|<r\}$, and
$$
Q_r:=B'_{r}\times (-r,r).
$$

We can now state and give a simple proof of Theorem 5.1 in \cite{CRS10}, 
which established that minimizers are viscosity solutions. Our statement, which is the same as that of \cite{CRS10}
but extended to other kernels besides the standard power, concerns
one-sided minimizers from outside (in our terminology) ---they are called supersolutions in \cite{CRS10}.
We use the notation $-e_n=(0,\ldots,0,-1)\in\R^n$.

\begin{theorem} \label{thm:visc}
Assume that $K$ satisfies \eqref{kernel} and $K(z)\leq C |z|^{-n-\alpha}$ for some constants $\alpha\in (0,1)$ and $C$.

Let $\Omega\subset\R^n$ be a bounded open set and $F\subset\R^n$ be a one-sided minimizer in $\Omega$ from outside, that is,
\begin{equation}\label{outside}
 \Pc_\Omega(F)\leq \Pc_\Omega(G) \quad\text{for all $G$ satisfying } F\subset G \text{ and } G\setminus\Omega=F\setminus\Omega.
\end{equation}
Assume in addition that $0\in \partial F$, $B_1(-e_n)\subset F$, and that $|B_r\setminus F|>0$ for all $r>0$ 
sufficiently small\;\footnote{ This last assumption is also made in Theorem 5.1 of \cite{CRS10}, as mentioned
in the beginning of Section~4 of that paper.}.
Then, the $K$-nonlocal mean curvature of $F$ at $0$ is well defined in the principal value sense and satisfies 
\begin{equation}\label{ineqvisc}
H_K[F](0)=\lim_{\varepsilon \downarrow 0} \int_{\R^n\setminus B_\varepsilon} \left(\cara_{F^c} (y)-\cara_{F}(y)\right) K(y)\, dy \geq 0.
\end{equation}

As a consequence, minimizers of the $K$-nonlocal perimeter in $\Omega$ are viscosity solutions in $\Omega$. 
\end{theorem}

The proof in \cite{CRS10} establishes that  \eqref{ineqvisc} is nonnegative when replacing the limit on it by a $\liminf$,
but not that the limit exists (i.e., that the principal value is well defined). Thus, we will include the proof of this fact below.

The main difficulty in the proof of Theorem~\ref{thm:visc} is to establish the following result, 
a typical viscosity statement.

\begin{proposition}\label{prop:visc}
 Under the hypothesis on the kernel $K$ of Theorem \ref{thm:visc}, 
assume \eqref{outside}, i.e., that $F$ is a one-sided minimizer in $\Omega$ from outside.

Suppose that $0\in \partial F$, $|B_r\setminus F|>0$ for all $r>0$ 
sufficiently small, $A\subset F$, and 
\begin{equation}\label{bdry}
 A \cap Q_2 = \{ x\in Q_2 :  x_n < u(x') \}
\end{equation}
for some smooth function $u: \overline{B'_2}\subset\R^{n-1}\to\R$ with $u(0)=0$ and $\nabla u(0)=0$. 
Assume also that for each $\rho\in (0,1)$ there exists a constant $\nu>0$ such that
\begin{equation}\label{dist}
{\rm dist} \left ( (x',u(x')), F^c \right) > \nu >0 \quad\text{ if } \rho\leq |x'|\leq 1.
\end{equation}

Then, $H_K[A](0) \geq 0$.
\end{proposition}

Note that here $H_K[A](0)$ is well defined since $A$ is smooth in a neighborhood of the origin.

With this result at hand, we con now give the

\begin{proof}[Proof of Theorem \ref{thm:visc}]
Given $\varepsilon\in (0,1)$, we consider the set
\begin{equation}\label{Feps}
A_\varepsilon:= B_{1/2} (- e_n/2) \cup \left( F\setminus B_\varepsilon  \right).
\end{equation}
Note that $0\in\partial A_\varepsilon$, $A_\varepsilon\subset F$, and $A_\varepsilon \setminus B_\varepsilon = F\setminus B_\varepsilon $. 
Since $\partial A_\varepsilon$ is smooth in a neighborhood of $0$,  $H_K [A_\varepsilon](0)$ is well defined in the principal value sense.

After a dilation, we can transform the sets $F$, $A_\varepsilon$, and an appropriate small
cylinder $Q_{c\varepsilon}$ centered at $0$, to become respectively three sets $F$, $A$, and $Q_2$ 
satisfying the hypotheses of Proposition~\ref{prop:visc}. Notice that we may assume \eqref{dist} since
the tangent ball $B_{1/2} (- e_n/2)$ in \eqref{Feps} separates
strictly from $B_1 (-e_n)$ away from the origin, while still $B_1 (-e_n)\subset F$.
We therefore have, by the above proposition,
$$
H_K [A_\varepsilon](0)\geq 0 \quad \text{ for each }\varepsilon>0.
$$

From this information we are going to conclude that $H_K [F](0)$ is well defined and nonnegative. For this, we adapt to our perimeter setting
the ideas for integro-differential operators of Caffarelli and Silvestre, Lemma~3.3 of \cite{CS09}.
Since $A_\varepsilon$ is nondecreasing as $\varepsilon$ decreases to $0$, the function
$$
f_\varepsilon := \cara_{(A_\varepsilon)^c} -  \cara_{A_\varepsilon}
$$
is nonincreasing as $\varepsilon\downarrow 0$. The same monotonicity holds for the function
$$
g_\varepsilon := \cara_{(A_\varepsilon)^c} -  \cara_{A_\varepsilon} +\cara_{B_{1/2} (- e_n/2)} -\cara_{B_{1/2} (e_n/2)},
$$
which will be used next. We also notice that $g_\varepsilon\leq 0$ in $B_{1/2} (-e_n/2)\cup B_{1/2} (e_n/2)$.
Thus, given $0<\delta<\varepsilon <1$, we have (note that all integrands are integrable since we remove a ball of radius $\delta$ about
the origin)
\begin{align}
\int_{\R^n\setminus B_\delta} f_\varepsilon (y) K(y)\, dy &=  \int_{\R^n\setminus B_\delta} g_\varepsilon (y) K(y)\, 
dy\label{eq1}\\
& \hspace{-3cm} =  \int_{\R^n\setminus \left (B_\delta \cup B_{1/2} (-e_n/2)\cup B_{1/2} (e_n/2)\right)} (g_\varepsilon)^+ (y)K(y)\, dy - 
\int_{\R^n\setminus B_\delta} (g_\varepsilon)^-(y)K(y)\, dy.\label{eq2}
\end{align}
We know that \eqref{eq1} tends to $H_K [A_\varepsilon](0)\geq 0$ as $\delta\downarrow 0$. On the other hand, since the first integral in
\eqref{eq2} is bounded above by 
$$
2  \int_{\R^n\setminus\left (B_{1/2} (-e_n/2)\cup B_{1/2} (e_n/2)\right) } K(y)\, dy<\infty
$$
---that this integral is finite follows from an explicit computation using the power upper bound for $K$---
it is also convergent as 
$\delta\downarrow 0$. We conclude that also the second integral in \eqref{eq2} is monotonically convergent 
to a finite value, as $\delta\downarrow 0$, and that
\begin{align}
0 & \leq  H_K [A_\varepsilon](0)= \int_{\R^n} f_\varepsilon (y) K(y)\, dy =  \int_{\R^n} g_\varepsilon (y) K(y)\, 
dy\label{eq3}\\
&  =  \int_{\R^n\setminus \left (B_{1/2} (-e_n/2)\cup B_{1/2} (e_n/2)\right)} (g_\varepsilon)^+ (y)K(y)\, dy - 
\int_{\R^n} (g_\varepsilon)^-(y)K(y)\, dy.\notag
\end{align}

Thus,
\begin{align}
\int_{\R^n} (g_\varepsilon)^-(y)K(y)\, dy& \leq \int_{\R^n\setminus 
\left (B_{1/2} (-e_n/2)\cup B_{1/2} (e_n/2)\right)} (g_\varepsilon)^+ (y)K(y)\, dy 
\notag\\
& \leq 2  \int_{\R^n\setminus\left (B_{1/2} (-e_n/2)\cup B_{1/2} (e_n/2)\right) } K(y)\, dy<\infty \notag
\end{align}
and,  as $\varepsilon\downarrow 0$, we have that $(g_\varepsilon)^+$ is nonincreasing, 
$(g_\varepsilon)^-$ is nondecreasing,
and $g_\varepsilon \to  g:=\cara_{F^c} -  \cara_{F} +\cara_{B_{1/2} (- e_n/2)} -\cara_{B_{1/2} (e_n/2)}$ a.e.\ in $\R^{n}$.
By monotone convergence we conclude that $g^+$ and $g^-$ are integrable in $\R^n$ and, from \eqref{eq3}, that $\int_{\R^n}g(y)K(y)\,dy\geq 0$.
But then
$$
\int_{\R^n\setminus B_r} g(y)K(y)\, dy
=\int_{\R^n\setminus B_r} \left( \cara_{F^c}(y) -  \cara_{F}(y)\right) K(y)\, dy
$$
has limit as $r\downarrow 0$ and the limit is nonnegative, as claimed.
\end{proof}

To prove the proposition establishing the viscosity statement, we will need the following observation.

\begin{lemma}\label{rk:reg}
Assume that $K$ satisfies \eqref{kernel} and $K(z)\leq C |z|^{-n-\alpha}$ for some constants $\alpha\in (0,1)$ and $C$.
Let $D_t$ be a family of measurable sets in $\R^n$ such that 
\begin{equation}\label{setseq}
D_t\cap Q_1 = D_0\cap Q_1 \qquad\text{for all } t\in [0,t_0]. 
\end{equation}
Assume that $D_0\cap Q_1$ is a $C^{1,\beta}$ open set for some $\beta>\alpha$, $0\in\partial D_0$, and such that
$\cara_{D_t}\to \cara_{D_0}$ in $L^ 1_{\rm loc} (\R^n)$ as $t\downarrow 0$. 

Then, the $K$-nonlocal mean curvature
$H_K[D_t](x)$, as a function of $(t,x)\in [0,t_0]\times (\partial D_t\cap Q_1)$, is continuous at $(t,x)=(0,0)$.
Furthermore, continuity at $(0,0)$ also holds for the function $H_{K_\varepsilon}[D_t](x)$ appearing in the principal value definition of
$H_K[D_t](x)$, with a uniform modulus of continuity for $\varepsilon$ small.
\end{lemma}
 
 As we will see in the proof, the statement could be made more general, replacing the identity  \eqref{setseq} by
 regular enough dependence of $D_{t}$ in $t$ within $Q_{1}$.
 
\begin{proof}
One breaks the integral defining $H_K[D_t](x)$ as follows:
\begin{align*}
H_K[D_t](x) & =\int_{\R^n} \left( \cara_{(D_t)^c}(y)-\cara_{D_t}(y)\right) \cara_{(Q_{1/2})^c}(x-y) K(x-y) \, dy\\
& \hspace{1cm} + \int_{\R^n} \left( \cara_{(D_t)^c}(x+z)-\cara_{D_t}(x+z)\right) \cara_{Q_{1/2}}(z) K(z) \, dz.
\end{align*}
We can restrict the first integral to a compact set since the integral of $K$ near infinity is as small as we wish. Since $x-y\in (Q_{1/2})^{c}$, the kernel is bounded in the compact set intersected with the complement of the cube. Thus, the hypothesis $\cara_{D_t}\to \cara_{D_0}$ in $L^ 1_{\rm loc} (\R^n)$
gives the continuity of the integral as $(t,x)\to(0,0)$. The same statements hold for the kernels $K_{\varepsilon}$, uniformly in $\varepsilon$.

Regarding the second integral, taking $x\in Q_{1/2}$ we have that all the integrands $\cara_{(D_t)^c}(x+z)-
\cara_{D_t}(x+z)=\cara_{(D_0)^c}(x+z)-\cara_{D_0}(x+z)$ coincide. Therefore we are now dealing with the nonlocal mean curvature 
of the set $D_0$ for the kernel $\cara_{Q_{1/2}} K$. 
Its continuity as a function of the boundary point is proved with all details in Proposition~6.3 of \cite{FFMMM15} 
under the only hypothesis on the power upper bound for the kernel, which is satisfied by $\cara_{Q_{1/2}} K$.
The proof in \cite{FFMMM15} also establishes the continuity of $H_{K_\varepsilon}[D_t](x)$ uniformly in $\varepsilon$.
Note that the $C^1$ assumption on $K$ (away from the origin) made in  \cite{FFMMM15} is not used in their proof
when establishing the continuity of the nonlocal mean curvature\footnote{ This continuity result has also been proved in other articles, but
we find the proof in \cite{FFMMM15} the most natural and direct ---it contains, though, a couple of ``easy-to-correct'' typos.}.
\end{proof}

We can now proceed to prove the viscosity property using the calibration functional, more precisely, 
using Theorem~\ref{formula}.

\begin{proof}[Proof of Proposition \ref{prop:visc}]
We argue by contradiction and assume that
\begin{equation}\label{contr}
H_K [A](0) < 0.
\end{equation}
The idea of the proof is to raise the graph of $u$ in the vertical direction ---just slightly to preserve the
exterior datum $F$ outside of a neighborhood of the origin, thanks to \eqref{dist}--- to produce a local foliation of the complement of $F$,
around the origin, 
with leaves having negative nonlocal mean curvature. Theorem~\ref{formula} will then give that the set below the last leaf 
(completed with $F$ so that the new set $E$ is larger than $F$ as required in Theorem~\ref{formula}) has
smaller $K$-nonlocal perimeter than $F$, a contradiction with $F$ being a minimizer from outside. Since $F$ may not be
smooth, we are looking at the foliation in an irregular set $E\setminus F$ and this why we need the generalized notion of nonlocal mean curvature from the previous section. 
However, adding $F$ to the smooth leaves raised from the graph of $u$ will only help to
make the nonlocal mean curvature even more negative.

More precisely, for $t\in [0,1]$, we define
\begin{equation*}\label{St}
A_t := A\cup \{ x\in \R^n : |x'|< 1, u(x')\leq x_n < u(x') +t \}.
\end{equation*}
We will always take $\rho\in (0,1/2)$ small enough to have $u<\rho/3$ in $B_\rho'$, and also $0\leq t \leq \rho/3$. 
Thanks to this and since
we have the equality \eqref{bdry} in $Q_2$ while $A_t$ only differs from~$A$ within $\{|x'|<1\}$, it is easy to check that
\begin{equation}\label{boundary}
\partial A_t\cap Q_{\rho}=\{(x',u(x')+t) : |x'|<\rho\}.
\end{equation}
On the other hand, \eqref{contr} and Lemma~\ref{rk:reg} give that, for $\rho$, $t$, and $\varepsilon$ small enough,
\begin{equation}\label{all-signs}
H_{K_\varepsilon} [A_t](x) < 0 \quad \text{and}\quad H_K [A_t](x) < 0 \quad\text{ for all } x\in \partial A_t\cap Q_{\rho}.
\end{equation}
To verify this, translate $A_t$ in the vertical direction so that all points on $\partial A_t\cap Q_{\rho}$ lie in the same
surface $\{x_n=u(x')\}$. Then $D_t:=A_t-te_n$ fits with the setting of Lemma~\ref{rk:reg}, since
$D_t=A_t-te_n=(A\cap Q_{1})\cup((A-te_{n})\cap (Q_{1})^{c})$.
Recall now that $\cara_{A-t e_n}\to\cara_A$ in $L^1_{\rm loc}(\R^n)$ as $t\to 0$ for each measurable set~$A$,
as shown by approximating in $L^{1}$ the characteristic function by continuous ones.

It is crucial for the sequel that, for $0<t<\nu$, we have
$$
A_t\setminus Q_\rho \subset F,
$$ 
since  $A\subset F$, $u(x')+t<2\rho/3<\rho$ in $Q_{\rho}$, and \eqref{dist}. 
Taking $t\in[0,t_0]$ and $\rho$, both small enough to guarantee all the previous statements, we define
\begin{equation}\label{defEt}
E_t := A_t\cup F \quad\text{ and }\quad E:=E_{t_0}.
\end{equation}
Therefore, we have
\begin{equation}\label{ending}
F\subset E_t \quad \text{and}\quad E_t\setminus Q_\rho =F\setminus Q_\rho,
\end{equation}
which fits with the setting of Theorem \ref{formula} by taking $\Omega=Q_\rho$ there. 
With the notation of the theorem, we define
$\phi(x)= t_0+u(x')-x_n$ for $x\in E\setminus F\subset Q_\rho$ and consider the function $\phi_E$ defined by \eqref{totalphi} 
(recall that we take $E=E_{t_0}$).

Note here that in general $E$ will not be smooth, since we are adding $F$ to the smooth set $A_{t_{0}}$
and $F$ is apriori an irregular set. This
is why in the previous section we had to introduce a generalized notion of nonlocal mean curvature and adapt the calibration to it.

Now, given $x\in E\setminus F$ we have, in particular, that $x\in (A_{t_0}\setminus A)\cap Q_\rho$
(by \eqref{defEt}, \eqref{ending}, and $A\subset F$). 
Letting $t=t(x):= t_0-\phi_E(x)=x_n-u(x')\in [0,t_0)$, we claim that
\begin{equation}\label{lastclaim}
 x\in\partial A_t\cap Q_\rho \; \text{ and }\;  \sign \left(\phi_E(x)-\phi_E(y)\right) \leq \cara_{(A_t)^c}(y)-\cara_{A_t}(y)
 \text{ for all } y\in \R^ n.
\end{equation}
Indeed, that $x\in\partial A_t$ is clear by \eqref{boundary}. We next check the claimed inequality.
From the definition of $\phi_E$, there is nothing to check when $y\in F$ ---recall that $\phi_{E}(x)=\phi(x)\in\R$. When $y\in E^c$ the inequality becomes an equality 
since $E^c\subset (E_t)^c\subset (A_t)^c$. Finally, for $y\in E\setminus F\subset (A_{t_0}\setminus A)\cap Q_\rho$, there is
nothing to check if $\phi_E(x)-\phi_E(y)<0$. Instead, if 
$0\leq \phi_E(x)-\phi_E(y)=y_n-u(y')-t$
then $y \in (A_t)^c$ and thus the inequality in \eqref{lastclaim} holds again.
Therefore, recalling the definitions \eqref{epsnmc} and  \eqref{gennmc}, from \eqref{lastclaim} and \eqref{all-signs} we deduce that
\begin{equation}\label{finaleps}
H_{K_\varepsilon}[\phi_E](x)\leq H_{K_\varepsilon}[A_t](x)<0
\end{equation}
and
\begin{equation}\label{final}
\tilde H_{K}[\phi_E](x)\leq H_{K}[A_t](x)<0.
\end{equation}

Thus, these inequalities hold for every $\varepsilon$ sufficiently small and $x\in E\setminus F$, with $t=t_0-\phi_E(x)$.  
In particular, by \eqref{finaleps}, the hypothesis \eqref{Hbdd} in Theorem~\ref{formula} is satisfied with $C_0=0$.
Note that since $F$ is a one-sided minimizer in $\Omega=Q_{\rho}$ from outside, we necessarily have 
$\Pc_{\Omega}(F)\leq \Pc_{\Omega}(F\cup \Omega)\leq \Lc (\Omega,\Omega^{c})
= \Lc (Q_{\rho}, (Q_{\rho})^{c})<+\infty$.
Thus, Theorem~\ref{formula} and \eqref{final} give
$$
\Pc_\Omega (E) \leq \Pc_\Omega (F)+ \int_{E\setminus F}  \tilde H_K [\phi_E](x)\, dx <\Pc_\Omega (F)
$$
provided that $|E\setminus F|>0$. This is a contradiction with $F$ being a one-sided minimizer in $\Omega$ from outside. 
Finally, note that $|E\setminus F|>0$ holds thanks to the assumption $|B_r\setminus F|>0$, for $r$ small enough, made in the proposition. 
Indeed, if $r$ is sufficiently small we have that $B_r\subset A_{t_0}$ and hence
$B_r\setminus F\subset E\setminus F$.
\end{proof}

\section*{Acknowledgments}

\noindent
The author thanks Joaquim Serra for several interesting conversations on the topic and for 
suggestions after reading a first version of the paper. 
He also thanks the referee for very interesting comments and for bringing to our attention
the related works \cite{CMP,LMaster}.

\end{document}